\documentclass[letterpaper, 10 pt, conference]{ieeeconf}  

\IEEEoverridecommandlockouts                              

\usepackage[letterpaper,top=54pt,bottom=58pt,left=54pt,right=54pt]{geometry}

\usepackage{graphics} 
\usepackage{epsfig} 
\usepackage{mathptmx} 
\usepackage{times} 
\usepackage{amsmath} 
\usepackage{amssymb}  
\usepackage{xcolor}
\usepackage{dsfont}

\usepackage{subcaption}
\usepackage{hyperref}

\newtheorem{remark}{Remark}
\newtheorem{theorem}{Theorem}

\newtheorem{definition}{Definition}
\newtheorem{assumption}{Assumption}
\newtheorem{lemma}{Lemma}
\newtheorem{prop}{Proposition}
\begin{document}

\title{\LARGE \bf {Data-driven identification of reaction-diffusion dynamics from finitely many non-local noisy measurements by exponential fitting}}
\author{Rami Katz, Giulia Giordano and Dmitry Batenkov
\thanks{R. Katz ({\tt\small ramkatsee@gmail.com}) and G. Giordano ({\tt\small   giulia.giordano@unitn.it}) are with the Department of Industrial Engineering, University of Trento.  }
\thanks{D. Batenkov ({\tt\small dbatenkov@tauex.tau.ac.il}) is with the Department of Applied Mathematics, Tel Aviv University, Israel.}
\thanks{The work of R. Katz and G. Giordano was funded by the European Union through the ERC INSPIRE grant (project number 101076926); views and opinions expressed are however those of the authors only and do not necessarily reflect those of the European Union or the European Research Council; neither the European Union nor the granting authority can be held responsible for them. D.~Batenkov was supported by Israel Science Foundation Grant 1793/20.}
}

\maketitle

\begin{abstract}
Given a reaction-diffusion equation with unknown right-hand side, we consider a nonlinear inverse problem of estimating the associated leading eigenvalues and initial condition modes from a finite number of non-local noisy measurements. We define a reconstruction criterion and, for a small enough noise, we prove the existence and uniqueness of the desired approximation and derive closed-form expressions for the first-order condition numbers, as well as bounds for their asymptotic behavior in a regime when the number of measured samples is fixed and the inter-sampling interval length tends to infinity. When computing the sought estimates numerically, our simulations show that the exponential fitting algorithm ESPRIT is first-order optimal, as its first-order condition numbers have the same asymptotic behavior as the analytic condition numbers in the considered regime.
\end{abstract}

\section{Introduction and Considered Model}
Reaction-diffusion equations (RDEs) are widely used to model phenomena in physics and engineering, including magnetized plasma, flame front propagation and chemical processes \cite{sivashinsky1977nonlinear,nicolaenko1986some}.  RDEs belong to the class of distributed parameters systems, and their control and observation have been investigated over the last decades, see e.g. \cite{balas1988finite,harkort2011finite,katz2022delayed}. In particular, observation and control of RDEs through modal decomposition was employed e.g. in \cite{christofides2001,curtain1982finite,katz2021finite111}. Almost all existing control and observation techniques assume explicit knowledge of the spatial operator of the system or of the eigenvalue/eigenfunction pairs corresponding to its modes.

Identification of unknown parameters in RDEs is a challenging problem, which has been mostly studied in an adaptive estimation framework \cite{demetriou1994dynamic,banks2012estimation}. Adaptive estimation relies on a persistency of excitation assumption, which may be difficult to verify in practice. It also requires continuous-time measurements of the state and has not been generalized so far to a sampled-data framework and/or to estimation from a \emph{finite} number of measurements. Finally, translation of these theoretical methods into tractable and efficient algorithms is, to the best of our knowledge, still an open problem. Other identification methods, which are accompanied by sound numerical algorithms, have been derived in the field of inverse problems \cite{lowe1992recovery,rundell1992reconstruction,kirsch2011introduction}. These approaches treat the problem of recovering the spatial operator of the system under the assumption of \emph{complete knowledge} of its eigenvalues. However, this assumption is non-realistic from a control theory perspective, since often only discrete-time measurements of the state are available. Hence, constructive and implementable data-driven identification techniques for reaction-diffusion equations are still missing.

We consider the 1D reaction-diffusion equation  
\begin{small}
\begin{equation}\label{eq:PDE}
\begin{array}{lll}
&\hspace{-2mm}z_t(x,t)= \left(p(x)z_x(x,t) \right)_x+q(x)z(x,t),\ z(0,t) = z(1,t) = 0,
\end{array}
\end{equation}
\end{small}

subject to non-local measurements

\begin{small}
\begin{equation}\label{eq:Meas}
 y(t) = \int_0^1c(x)z(x,t)dx\in \mathbb{R},\quad t \geq 0.   
\end{equation}
\end{small}

Here $c\in L^2(0,1)$ is \underline{partially known} (see Assumption $1.2$ in Section \ref{Sec:IDObject}), $x\in (0,1)$ and $z(x,t)\in \mathbb{R}$. The smooth functions $p(x)$ and $q(x)$ and the initial condition $z(\cdot,0)$ are assumed \underline{unknown}. The system \eqref{eq:PDE} has an associated sequence of eigenvalues $\left\{\lambda_n\right\}_{n=1}^{\infty}$ with eigenfunctions $\left\{\psi_n \right\}_{n=1}^{\infty}$ (see Section \ref{Sec:Prelim}). The identification objective considered in this paper is the approximation of the leading eigenvalues $\left\{\lambda_n\right\}_{n=1}^{N_0},\ N_0\in \mathbb{N}$, and the initial condition modes from the available measurements, subject to appropriate assumptions (see Section \ref{Sec:IDObject}). The contribution of the paper is as follows:
\begin{enumerate}
    \item Differently from existing adaptive estimation algorithms, which require measurements of the form $y(t)$, $t\geq 0$, or $\left\{y(s_k)\right\}_{k=1}^{\infty}$ with $\lim_{k\to \infty}s_k=\infty$, we assume that the measurements \eqref{eq:Meas} are available at \underline{finitely many} time steps. Moreover, $c$ in \eqref{eq:Meas} is not a perfect filter (meaning, $c\notin \operatorname{span}\left\{\psi_n \right\}_{n=1}^{N_1}$) and introduces structured noise into the measurements, with intensity $\epsilon$, which emanates from measuring `undesirable' system modes.
    \item We define a reconstruction criterion in the presence of structured noise, and prove the existence and uniqueness of the associated approximation, provided $\epsilon>0$ is not too large (see Theorem \ref{Prop:IFT}). 
    \item We introduce first-order condition numbers, in \eqref{eq:EpsExpans}, which describe how the $\epsilon$-noise is amplified in the reconstruction errors, and provide explicit expressions for these condition numbers, as well their asymptotic behavior in the specific regime described in \eqref{eq:Regimes}.
    \item Finally, we consider the problem of numerically computing the approximations. The parameter identification problem turns out to be a special case of \emph{exponential fitting}, a classical topic in data analysis with numerous applications \cite{istratov1999a,pereyra2010,batenkov2021a}. Our numerical simulations show that the well-known ESPRIT algorithm \cite{roy1989esprit} achieves first-order optimality, meaning that the first-order condition numbers of the ESPRIT algorithm exhibit the same asymptotic behaviour as the analytic condition numbers, in the considered regime. 
\end{enumerate}
We believe that our results pave the way towards new directions in data-driven identification of RDEs.

\section{Preliminaries}\label{Sec:Prelim}
Consider the system \eqref{eq:PDE}, where the \emph{unknown} smooth functions  satisfy the bounds
\begin{small}
\begin{equation}\label{eq:pbounds}
\begin{array}{lll}
&0<\underline{p}\leq p(x) \leq \overline{p} <\infty,\ \underline{q}\leq q(x) \leq \overline{q}, \quad  x\in [0,1].
\end{array}
\end{equation}
\end{small}

The constants $\underline{p}, \underline{q},\overline{p}, \overline{q}$ are assumed to be \emph{unknown}. We denote by $\mathcal{H}^2(0,1)$ (resp. $\mathcal{H}_0^1(0,1)$) the Sobolev space of functions $f$ defined on $[0,1]$ that are twice (resp. once) weakly differentiable with $f''\in L^2(0,1)$ (resp. $f'\in L^2(0,1)$ and $f(0)=f(1)=0$). Define the operator $\mathcal{A}$ 
\begin{small}
\begin{equation*}
\begin{array}{lll}
&\left[\mathcal{A}h\right](x) = -\left(p(x)h'(x) \right)'-q(x)h(x),\quad x\in(0,1),\\
& \operatorname{Dom}(\mathcal{A}) = \left\{h\in \mathcal{H}^2(0,1) \ ; \ h(0)=h(1)=0 \right\}.
\end{array}
\end{equation*}
\end{small}

The operator $\mathcal{A}$ has an infinite monotone sequence of simple eigenvalues $\left\{\lambda_n\right\}_{n=1}^{\infty}$ satisfying $\lim_{n\to \infty}\lambda_n=\infty$ \cite{orlov2017general}. The eigenvectors $\left\{\psi_n \right\}_{n=1}^{\infty}$ form a complete orthonormal system in $L^2(0,1)$. Also, they satisfy the inequalities \cite{orlov2017general}
\begin{small}
 \begin{equation}\label{eq:EigBounds}
   \pi^2n^2\underline{p}+\underline{q} \leq \lambda_n \leq  \pi^2 n^2 \overline{p}+\overline{q},\quad n\in \mathbb{N}.
 \end{equation}
\end{small}

Well-posedness of system \eqref{eq:PDE} has been studied thoroughly \cite{katz2021finite}. In particular, given $z(\cdot,0)\in L^2(0,1)$, system \eqref{eq:PDE} has a unique solution $z\in C([0,\infty), \mathcal{H}^1_0(0,1))\cap C^1((0,\infty), \mathcal{H}^1_0(0,1))$ such that $z(\cdot,t )\in \operatorname{Dom}(\mathcal{A})$ for all $t>0$.  We present the solution to \eqref{eq:PDE} as
\begin{equation}\label{eq:Series}
z(x,t) = \sum_{n=1}^{\infty}z_n(t)\psi_n(x),\quad z_n(t) = \left<z(\cdot,t),\psi_n \right>, \ n\in \mathbb{N}.
\end{equation}
Differentiating under the integral sign and integrating by parts, we have $\dot{z}_n(t) = -\lambda_nz_n(t)$ for all $n\in \mathbb{N}$, whence
\begin{small}
\begin{equation}\label{eq:Series1}
z(x,t) = \sum_{n=1}^{\infty}z_n(0)e^{-\lambda_nt}\psi_n(x).
\end{equation}
\end{small}

Henceforth, we adopt the notation $[n] = \{i \in \mathbb{N} \ ; \ 1 \leq i \leq n \}$.
\begin{prop}\label{prop:Differencebounds}
There exist constants $\upsilon,\Upsilon>0$ such that
\begin{small}
\begin{equation}\label{eq:EigDiff}
  \upsilon \left(m^2-n^2 \right) \leq \lambda_m-\lambda_n\leq \Upsilon \left(m^2-n^2 \right)
\end{equation}
\end{small}
holds for every choice of $1\leq n < m$.
\end{prop}
\begin{proof}
We first show that \eqref{eq:EigDiff} holds for $A\leq n< m$, with some $A\in \mathbb{N}$. By \cite[Equation 4.21]{fulton1994eigenvalue}, the eigenvalues $\left\{\lambda_n\right\}_{n=1}^{\infty}$ have the asymptotic behavior $\lambda_n = \frac{\pi^2}{B^2}n^2+a_0+O\left(\frac{1}{n^2}\right)$, $n\geq 1$, where $B$ and $a_0$ are \emph{positive} constants. Hence, $\lambda_m-\lambda_n \geq  \frac{\pi^2}{2B^2}\left(m^2-n^2 \right)+\frac{\pi^2}{2B^2} + O\left(\frac{1}{A^2} \right)$,
which implies the lower bound in \eqref{eq:EigDiff}, for large enough $A$. The upper bound is proved via similar arguments.
Next, given $A$, it is clear that by increasing $\Upsilon$ and decreasing $\nu$, we can further guarantee \eqref{eq:EigDiff} for $1\leq n <m\leq A$. We show that $\Upsilon$ and $\upsilon$ can be tuned such that \eqref{eq:EigDiff} holds for $1\leq n\leq A$ and $m>A$. Assume $\nu$ cannot be found such that \eqref{eq:EigDiff} holds. Then, for any $q\in \mathbb{N}$, there exist $n_q\leq A, \ m_q>A$ such that
\begin{small}
\begin{equation*}
2^{-q}>\frac{\lambda_{m_q}-\lambda_{n_q}}{m_q^2-n_q^2}\overset{\eqref{eq:EigBounds}}{\geq} \frac{\pi^2\left(\underline{p}m_q^2 -\overline{p}A^2\right)+\underline{q}-\overline{q}}{m_q^2-1}.
\end{equation*}
\end{small}
By $\lim_{q \to \infty}m_q= \infty$ we get a contradiction. Similar arguments hold for $\Upsilon$.
\end{proof}

For $\chi=\left\{\chi_n \right\}_{n=1}^{S}\subseteq \mathbb{C}$, the Lagrange interpolation basis is
\begin{small}
\begin{equation}\label{eq:Lagrange}
    L_{\chi,n}(z) =\prod_{j\neq n}\frac{z-\chi_{j}}{\chi_n-\chi_{j}},\quad n \in [S],
\end{equation}
\end{small}
and the corresponding Hermite interpolation basis is
\begin{small}
\begin{equation}\label{eq:Hermite}
\begin{array}{lll}
&H_{\chi,n}(z) = \left[ 1-2(z-\chi_{n})L_{\chi,n}'(\chi_{n})\right] L_{\chi,n}^2(z) ,\\
&\widetilde{H}_{\chi,n}(z) = (z-\chi_{n})L_{\chi,n}^2(z), \quad n\in [S].
\end{array}
\end{equation}
\end{small}
The Hermite basis contains polynomials of degree at most $2S-1$. For a polynomial $q(z) = \sum_{j=0}^{S}a_jz^j$, we introduce the coordinate map $\mathfrak{C}(q) =  \operatorname{col}\left\{ a_j\right\}_{j=0}^{S}$.  The Hermite matrix $\mathbb{H}_S\left(\chi\right)  = \left( \operatorname{row}\left\{\begin{bmatrix} \mathfrak{C}(H_{\chi,n})&  \mathfrak{C}(\tilde{H}_{\chi,n}) \end{bmatrix}\right\}_{n=1}^S \right)^{\top}\in \mathbb{R}^{2S\times 2S}$ is the unique matrix satisfying 
\begin{small}
\begin{equation}\label{eq:Hermite1}
\mathbb{H}_S\left(\chi\right) \operatorname{col}\left\{ \zeta^j \right\}_{j=0}^{2S-1} = \operatorname{col}\left\{\begin{bmatrix}
H_{\chi,n}(\zeta)\\
\tilde{H}_{\chi,n}(\zeta)
\end{bmatrix} \right\}_{n=1}^S.
\end{equation}
\end{small}

\begin{prop}\label{prop:Ical}
Let $ 0\leq w_1< w_2\leq \infty$ and $\mathcal{Q}_{\alpha}(x) = -\log \left(1-e^{-\alpha x} \right)$ with $x>0$ and $\alpha>0$. The integrals  
\begin{small}
\begin{equation}\label{eq:CalInteg}
\mathcal{J}_{w_1,w_2}(\alpha) := \int_{w_1}^{w_2} \mathcal{Q}_{\alpha}(x)\mathrm{d}x>0
\end{equation}
\end{small}
are finite and decreasing, and $\lim_{\alpha \to \infty} \mathcal{J}_{w_1,w_2}(\alpha)=0$.
\end{prop}
\begin{proof}
We prove the result for $w_1=1$ and $w_2=\infty$; other cases are similar. Integrating by parts, we have
\begin{small}
\begin{equation*}
\mathcal{J}_{1,\infty}(\alpha) = \left[-x\log\left(1-e^{-\alpha x} \right) \right]_1^{\infty}+\alpha \int_1^{\infty}\frac{x}{1-e^{-\alpha x}}e^{-\alpha x}\mathrm{d}x.
\end{equation*}
\end{small}
Since $\lim_{x\to \infty}x\log\left(1-e^{-\alpha x} \right)  = 0$ and the integral on the right-hand side converges,
we have that $\mathcal{J}_{1,\infty}(\alpha)<\infty$. Given $x\in (0,\infty)$,  $(0,\infty)\ni \alpha \mapsto -\log \left(1-e^{-\alpha x} \right)\in (0,\infty)$ is decreasing, whence $\mathcal{J}_{1,\infty}(\alpha)$ is decreasing too. Let $\epsilon>0$, $\alpha>1$ and $M>1$. Then, $\mathcal{J}_{1,\infty}(\alpha) \leq \mathcal{J}_{1,M}(\alpha)+ \mathcal{J}_{M,\infty}(1)$. 
Choosing $M$ so that the rightmost term is smaller than $\frac{\epsilon}{2}$ and then, by Dini's theorem, $\alpha_*$ such that  $\mathcal{J}_{1,M}(\alpha)<\frac{\epsilon}{2}$ for $\alpha>\alpha_*$, we have that  $\alpha>\min (1,\alpha_*)$ implies
$0<\mathcal{J}_{1,\infty}(\alpha)<\epsilon$.
\end{proof}


\section{The Identification Objective}\label{Sec:IDObject}

\subsection{Measurements and standing assumptions}
Considering system \eqref{eq:PDE} subject to the non-local measurements \eqref{eq:Meas}, we
substitute \eqref{eq:Series1} into \eqref{eq:Meas} to obtain 
\begin{small}
\begin{equation}\label{eq:Meas1}
 y(t) = \sum_{n=1}^{\infty}c_nz_n(0)e^{-\lambda_nt},\quad c_n = \left<c,\psi_n \right>,\ n\in \mathbb{N}.  
\end{equation}
\end{small}
\begin{remark}
We \emph{do not assume} that the solution to \eqref{eq:PDE} is exponentially stable. In fact, for $\left|\overline{q} \right|>0$ large, the solution \eqref{eq:Series1} may contain (finitely) many unstable modes.
\end{remark}

Before formally stating our identification objective, we present our main assumptions on system \eqref{eq:PDE} and measurements \eqref{eq:Meas}. Let there exist $N_1, N_2\in \mathbb{N}$ such that the following properties hold.

\begin{assumption}\label{assum:coeff}
The coefficients $\left\{ c_n \right\}_{n=1}^{\infty}$ satisfy
\begin{enumerate}
    \item   $c_n = 0$ for all $n>N_1+N_2$,
    \item   $c_n$ are \emph{known and nonzero} for $n\in [N_1]$, 
    \item   $ \frac{|c_k|}{|c_n|}\leq M_c \epsilon$ for all $n\in [N_1], k\in [N_1+N_2]\setminus [N_1]$
\end{enumerate}
for some $\epsilon>0$ and  $M_c>0$.
We define
\begin{small}
\begin{equation*}
\tilde{c}_k := \frac{c_k}{\epsilon},\ k\in [N_1+N_2]\setminus [N_1].
\end{equation*}
\end{small}
\end{assumption}

\begin{assumption}\label{assum:incon}
The initial condition $z(\cdot,0) \in L^2(0,1)$ is \emph{unknown} and satisfies $z_n(0)\neq 0$ for $n\in [N_1]$. Furthermore, $   \frac{|z_k(0)|}{|z_n(0)|}\leq M_z$
for some $M_z> 0$, $n\in [N_1]$ and $k\in [N_1+N_2]\setminus [N_1]$.
\end{assumption}

\begin{remark}
Our proposed approach allows to approximate $\left\{y_n \right\}_{n=1}^{N_1}$ even if  both $\left\{c_n \right\}_{n=1}^{N_1}$ and $\left\{z_n(0) \right\}_{n=1}^{N_1}$ are unknown. However, knowledge of $\left\{y_n \right\}_{n=1}^{N_1}$ in \eqref{eq:yDef} does not allow to separate $\left\{c_n\right\}_{n=1}^{N_1}$ and $\left\{z_n(0)\right\}_{n=1}^{N_1}$ from their (corresponding) products, if both sets are unknown. This is an inherent ambiguity of the identification objective.
In Assumptions \ref{assum:coeff} and \ref{assum:incon}, we consider $\left\{c_n\right\}_{n=1}^{N_1}$ to be known and $\left\{z_n(0)\right\}_{n=1}^{N_1}$ to be unknown. Our approach is also valid for unknown $\left\{c_n\right\}_{n=1}^{N_1}$  and known $\left\{z_n(0)\right\}_{n=1}^{N_1}$.
\end{remark}

\begin{assumption}\label{assum:meas}
The system measurements are available \emph{only at a finite set of times} $\left\{t_k := k\Delta\right\}_{k=0}^{2N_1-1}$, with step-size $\Delta>0$.
\end{assumption}

Subject to Assumptions \ref{assum:coeff}-\ref{assum:meas}, the measurements \eqref{eq:Meas1} at the available times $\left\{t_k \right\}_{k=0}^{2N_1-1}$ can be presented as 
\begin{small}
\begin{equation}\label{eq:Meas2}
 y(t_k) = \underbrace{\sum_{n=1}^{N_1}y_ne^{-\lambda_n \Delta k }}_{y_{\text{main}}(t_k)}+\epsilon \underbrace{\sum_{n=N_1+1}^{N_1+N_2}y_ne^{-\lambda_n\Delta k}}_{y_{\text{tail}}(t_k)}, \ k=0,\dots,2N_1-1
\end{equation}
\end{small}
where 
\begin{small}
\begin{equation}\label{eq:yDef}
 y_n : = \begin{cases}
     c_nz_n(0), \quad n\in  [N_1]\\
     \tilde{c}_nz_n(0),\quad n\in [N_1+N_2]\setminus [N_1]
 \end{cases} 
\end{equation}
\end{small}
satisfy $\frac{|y_k|}{|y_n|}\leq M_c M_z=:M_y$  
for all $n\in [N_1], k\in [N_1+N_2]\setminus [N_1]$.  

\subsection{Identification objective and considered regime}
We use the measurements \eqref{eq:Meas2} to estimate $\left\{z_n(0),\lambda_n \right\}_{n=1}^{N_0}$, for some $N_0\leq N_1$. To define a recovery criteria, we introduce
\begin{small}
\begin{equation}\label{eq:ImplicitFunc}
\mathcal{F}\left(\left\{\hat{y}_n,\hat{\lambda}_n \right\}_{n=1}^{N_1}, \epsilon  \right) = \operatorname{col}\left\{\sum_{n=1}^{N_1}\hat{y}_ne^{-\hat{\lambda}_n\Delta k}- y(t_k)\right\}_{k=0}^{2N_1-1} ,
\end{equation}
\end{small}
which reflects the discrepancy between the measurements $\left\{y(\Delta k) \right\} $ and the ``virtual measurements'' $\left\{\sum_{n=1}^{N_1}\hat{y}_ne^{-\hat{\lambda}_n\Delta k} \right\} $ obtained from a candidate $\hat{P}:=\left\{\hat{y}_n,\hat{\lambda}_n \right\}$. Note that $y(t_k)$ in \eqref{eq:Meas2} depends on $\epsilon$.
\begin{definition}\label{def:epsApprox}
$\hat{P}$ is an \emph{$\epsilon$-approximation} of $ \left\{y_n,\lambda_n \right\}$ if $ \mathcal{F}\left(\hat{P};\epsilon  \right) =0$.
\end{definition}
The intuition behind Definition \ref{def:epsApprox} is as follows:
\begin{itemize}
\item Measurements \eqref{eq:Meas2} are split into  $y_{\text{main}}$ and $y_{\text{tail}}$. The former contains the desired $\left\{z_n(0),\lambda_n \right\}_{n=1}^{N_0}$, whereas the latter is an \emph{$\epsilon$-structured perturbation}, where $\epsilon>0$, since $c\in L^2(0,1)$ is not a perfect filter, i.e. $c\notin \operatorname{span}\left\{\psi_n \right\}_{n=1}^{N_1}$.
\item When $c$ is a perfect filter, $\epsilon=0$, and thus $\mathcal{F}\left(P;0  \right) =0$, where $P = \left\{y_n,\lambda_n \right\}$. In particular, the $\epsilon$-approximation coincides with the true parameter, namely $P=\hat{P}$.
\item Hence, we propose to seek an approximation  $\hat{P}$ which preserves the equality  $\mathcal{F}\left(\hat P;\epsilon  \right) =0$ for $\epsilon>0$. 
\end{itemize}
Given an $\epsilon$-approximation $\hat{P}$, Assumption \ref{assum:coeff} and expression \eqref{eq:yDef} allow us to derive estimates for the dominant $N_0\leq N_1$ projection coefficients of the initial condition $z(\cdot,0)$, $\left\{z_n(0) \right\}_{n=1}^{N_0}$, since $\hat{z}_n(0) = \frac{\hat{y}_n}{c_n},\quad n\in [N_0]$.

We  address the analytical problem of existence of $\hat{P} = \hat{P}(\epsilon)$. Specifically:

1) Given a small $\epsilon>0$, we show that there exists a unique $\hat{P}=\hat{P}(\epsilon)$ and we derive the first-order expansions of the reconstruction errors with respect to $\epsilon$:
\begin{small}
\begin{equation}\label{eq:EpsExpans}
\begin{array}{lll}
&\hspace{-3mm}\frac{e_{\lambda}(n)}{\epsilon} := \frac{\hat{\lambda}_n(\epsilon) -  \lambda_n}{\epsilon}  =  \mathcal{K}_{\lambda}(n) + o(1), \\
&\hspace{-3mm}\frac{e_{y}(n)}{\epsilon}:= \frac{\hat{y}_n(\epsilon) - y_n}{\epsilon} =  \mathcal{K}_{y}(n)+ o(1),\ n\in [N_0], \ \epsilon\to 0,
\end{array}
\end{equation}
\end{small}
where $\mathcal{K}_{\lambda}(n)$ and $\mathcal{K}_{y}(n)$ are the \emph{first-order condition numbers} for the recovery of $\hat{P}$, which describe how much the  $\epsilon$-structured perturbation is amplified when estimating the reconstruction errors $e_{\lambda}(n)$ and $e_{y}(n)$.

2) Assuming $N_2$ fixed, we derive  the asymptotics of $ \mathcal{K}_{\lambda}(n)$ and $\mathcal{K}_{y}(n)$ for $n\in [N_1]$ in the
\begin{equation}\label{eq:Regimes}
\begin{array}{lll}
&\underline{\text{Regime:}} \ N_1 \text{ fixed and } \Delta \to \infty
\end{array}
\end{equation}
in order to obtain insight into the behaviour of the reconstruction errors  $e_{\lambda}(n)$ and $e_{y}(n)$ for small $\epsilon$.

3) When computing $\hat{P}$, our numerical simulations show that the ESPRIT algorithm \cite{roy1989esprit} is first-order optimal, i.e., its first-order condition numbers exhibit the same asymptotic behaviour as $\mathcal{K}_{\lambda}(n)$ and $\mathcal{K}_{y}(n)$ in the regime \eqref{eq:Regimes}. 
\begin{remark}\label{Rem:NonlinearInverse}
The considered problem is highly challenging for two reasons. First, we assume that only  \emph{finitely many} measurements are available for the reconstruction procedure, for any triplet $(\Delta,N_1,N_2)$. Second, although  \eqref{eq:PDE} is a linear system, the task of recovering $\left\{y_n,\lambda_n \right\}_{n=1}^{N_0}$ from the measurements \eqref{eq:Meas2} is a \emph{nonlinear inverse problem}, as the measurements depend \emph{nonlinearly} on these parameters.
\end{remark}

\section{Identification Problem: First-Order Analysis}\label{Sec:FirstOrdAn}

Here we derive explicit expressions for the first-order condition numbers $\mathcal{K}_{\lambda}(n)$ and $\mathcal{K}_{y}(n)$ in \eqref{eq:EpsExpans}, and bound their asymptotic behavior in the regime \eqref{eq:Regimes}.

Only to keep the presentation simpler, we assume that $N_2=1$ (i.e., the sum $y_{\text{tail}}(t_k)$, $k\in  \left\{0 \right\} \cup [2N_1-1]$, in \eqref{eq:Meas2} contains a single term). The analysis and the conclusions of this section remain identical for an arbitrary fixed $N_2\in \mathbb{N}$.

Throughout the section, we use the notation
\begin{small}
\begin{equation}\label{eq:Phidef}
 \phi_n := e^{-\lambda_n \Delta},\  n\in [N_1+1].   
\end{equation}
\end{small}
The measurements in \eqref{eq:Meas2} are then rewritten as 
\begin{small}
\begin{equation}\label{eq:Meas2A}
 y(t_k) = \sum_{n=1}^{N_1}y_n \phi_n^k+\epsilon y_{N_1+1}\phi_{N_1+1}^k, \ k\in  \left\{0 \right\} \cup [2N_1-1].
\end{equation}
\end{small}

We now prove the existence of an $\epsilon$-approximation $\hat{P}(\epsilon)$ and derive closed-form expressions for $\mathcal{K}_{\lambda}(n)$ and $\mathcal{K}_{y}(n)$.

\begin{theorem}\label{Prop:IFT}
There exist $\epsilon_*>0$ and continuously differentiable functions $\hat{P}(\epsilon):=\left\{\hat{y}_n(\epsilon),\hat{\lambda}_n(\epsilon) \right\}$ such that $\hat{P}(0) = P$ and that, for all $|\epsilon|< \epsilon_*$, $ \mathcal{F}(\hat{P};\epsilon)=0 \iff \hat{P}=\hat{P}(\epsilon)$.
Furthermore, we have 
\begin{small}
\begin{equation}\label{eq:Deriv}
\begin{bmatrix}
\mathcal{K}_{y}(n) \\ \mathcal{K}_{\lambda}(n)  
\end{bmatrix}
 = y_{N_1+1} 
\begin{bmatrix}
 H_{\Phi,n}(\phi_{N_1+1})\\
-\frac{1}{\Delta y_{n} \phi_n } \tilde{H}_{\Phi,n}(\phi_{N_1+1})
\end{bmatrix}, \ n\in [N_1],
\end{equation}
\end{small}
where $\left\{H_{\Phi,n},\tilde{H}_{\Phi,n} \right\}_{n=1}^{N_1}$ are the Hermite interpolation basis  polynomials, given in \eqref{eq:Hermite}, associated with $\Phi=\left\{\phi_n\right\}_{n=1}^{N_1}$.
\end{theorem}
\begin{proof}
Consider the function $\mathcal{F}(\hat{P}, \epsilon)$ in \eqref{eq:ImplicitFunc} and recall that $P=\left\{y_n, \lambda_n \right\}$ are the true parameters. It can be seen that  $\mathcal{F}$ is differentiable in all variables $(\hat{P}, \epsilon)$. We denote by $\partial_{\hat{P}} \mathcal{F}(P,0)$ its Jacobian with respect to $\hat{P}$ evaluated at $\hat{P}=P$ and $\epsilon=0$. Then, $\partial_{\hat{P}} \mathcal{F}(P,0)= J(P,0) D(P,0)$, with
\begin{small}
\begin{equation}\label{eq:J1Def}
J (P,0) = \mathbb{H}_{N_1}\left(\Phi\right)^{-1}, \quad D(P,0) = \operatorname{diag}\left\{ \Big[\begin{smallmatrix}
    1 & 0\\
    0 & -\Delta y_n\phi_n
 \end{smallmatrix} \Big] \right\}_{n=1}^{N_1},
\end{equation}
\end{small}
where $\mathbb{H}_{N_1}\left(\Phi\right)$ is the Hermite matrix, given in \eqref{eq:Hermite1}, associated with $\Phi$. Since the eigenvalues $\left\{\lambda_n \right\}_{n=1}^{\infty}$ are simple, it follows from the uniqueness of Hermite interpolation that $\mathbb{H}_{N_1}\left(\Phi\right)$ is invertible. In view of Assumptions \ref{assum:coeff}-\ref{assum:incon} and of \eqref{eq:yDef}, we have that $y_n\neq 0,\ n\in [N_1]$, whence  $\operatorname{det}(\partial_{\hat{P}} \mathcal{F}(P,0))\neq 0$. The implicit function theorem \cite{spivak2018calculus} guarantees that there exist $\epsilon_*>0$ and continuously differentiable functions $\hat{P}(\epsilon)$ such that $\hat{P}(0) = P$ and  $ \mathcal{F}(\hat{P};\epsilon)=0 \iff \hat{P}=\hat{P}(\epsilon)$ for all $|\epsilon| <\epsilon_*$. Differentiating $\mathcal{F}(\hat{P}(\epsilon),\epsilon)=0$ with respect to $\epsilon$ and substituting $\epsilon=0$, we obtain
\begin{small}
\begin{equation}\label{eq:Deriv1}
\begin{array}{lll}
&\hspace{-5mm} \operatorname{col}\left\{
\begin{bmatrix}
\mathcal{K}_{y}(n) \\ \mathcal{K}_{\lambda}(n)  
\end{bmatrix}
\right\}_{n=1}^{N_1} =\partial_{\hat{P}} \mathcal{F}(P,0)^{-1}y_{N_1+1}\operatorname{col}\left\{\phi_{N_1+1}^k\right\}_{k=0}^{2N_1-1}.
\end{array}
\end{equation}
\end{small}
Recalling that $\partial_{\hat{P}} \mathcal{F}(P,0)= \mathbb{H}_{N_1}\left(\Phi\right)^{-1}   D(P,0)$ and that \eqref{eq:Hermite1} holds, we obtain the expression in \eqref{eq:Deriv}.
\end{proof}
Theorem \ref{Prop:IFT} allows us to analyze the asymptotic behaviour of the condition numbers $\left\{\mathcal{K}_{y}(n)\right\}_{n=1}^{N_1} $ and $\left\{\mathcal{K}_{\lambda}(n)\right\}_{n=1}^{N_1}$ in the regime \eqref{eq:Regimes} through the analysis of the  polynomials in \eqref{eq:Deriv}.
Given $\Delta>0$, $N_1\in \mathbb{N}$ and $n\in [N_1]$, set the functions
\begin{small}
\begin{equation}\label{eq:XiComp}  
\begin{array}{lll}
&\hspace{-3mm} \xi_1 = \prod_{j\neq n}  (\phi_{N_1+1}-\phi_j)^2, \ \ \xi_2 = \prod_{j=1}^{n-1}(\phi_n-\phi_j)^2\\
&\hspace{-3mm} \xi_3 = \prod_{j=n+1}^{N_1}(\phi_n-\phi_j)^2, \ \ \xi_4 = \sum_{k \neq n}|\phi_n-\phi_k|^{-1},
\end{array}
\end{equation}
\end{small}
where all summations/products range over indices in $[N_1]$, and we use the convention that $\prod_{j=l}^k b_j=1$ and $\sum_{j=l}^k b_j=0$ whenever $k<l$.
We omit the dependence of functions on $(n,N_1,\Delta)$ for simplicity of notation.

\begin{remark}\label{rem:Remark4}
Recalling the Lagrange polynomials given in \eqref{eq:Lagrange} we observe that  $L_{\Phi,n}^2(\phi_{N_1+1}) = \frac{\xi_1}{\xi_2 \xi_3}$ and $\left| L_{\Phi,n}'(\phi_n)\right| = \xi_4$.
\end{remark}

To prove our main result, we need several lemmas.
\begin{lemma}\label{lem:ComponentBounds0}
The functions in \eqref{eq:XiComp} can be written as
\begin{small}
\begin{equation*}
\begin{array}{lll}
&\xi_1 = e^{-2\Delta \sum_{j\neq n}\lambda_j -2\theta_1 }, \quad   \xi_2 = e^{-2\Delta \sum_{j=1}^{n-1}\lambda_j - 2\theta_2} ,\quad n> 1, \\
& \xi_3 = e^{-2\Delta (N_1-n)\lambda_n - 2\theta_3} , \ n< N_1,
\end{array}
\end{equation*}
\end{small}
where
\begin{small}
\begin{equation*}
\begin{array}{lll}
&\hspace{-3mm}\theta_1:= \sum_{j\neq n}-\log \left(1-e^{-\Delta (\lambda_{N_1+1}-\lambda_j)}  \right)>0,\\
&\hspace{-3mm} \theta_2:=\sum_{j=1}^{n-1}-\log \left(1-e^{-\Delta (\lambda_n-\lambda_j)} \right)>0,\quad n>1,\\
&\hspace{-3mm} \theta_3:=\sum_{j=n+1}^{N_1}-\log \left(1-e^{-\Delta (\lambda_j-\lambda_n)} \right)>0,\ n<N_1
\end{array}
\end{equation*}
\end{small}
satisfy the inequalities
\begin{small}
\begin{equation*}
\begin{array}{lll}
& \theta_1 \leq  \mathcal{J}_{0,\infty}(\Delta \upsilon N_1)+\log \left(1-e^{-\Delta \upsilon (N_1+1-n)(N_1+1)}  \right) ,\\
& \theta_1\geq \mathcal{J}_{1,2}(\Delta \Upsilon (2N_1+1))+\log \left(1-e^{-\Delta \Upsilon (N_1+1-n)(2N_1+1)}  \right),   \\
&\mathcal{J}_{1,2}(\Delta \Upsilon (2n-1)) \leq \theta_2  \leq  \mathcal{J}_{0,\infty}(\Delta \upsilon (n+1)),\\
& \mathcal{J}_{1,2}(\Delta \Upsilon(N_1+n+1)) \leq \theta_3 \leq \mathcal{J}_{0,\infty}(\Delta \upsilon(2n+1)),
\end{array}
\end{equation*}
\end{small}
where the positive constants $\upsilon$ and $\Upsilon$ are those given in Proposition \ref{prop:Differencebounds} and the function $\mathcal{J}_{w_1,w_2}$ is given in \eqref{eq:CalInteg}.
\end{lemma}
\begin{proof}
Due to space constraints, we consider $\xi_1$. The results for $\xi_2$ and $\xi_3$ are proved similarly. We have
\begin{small}
\begin{equation*}
\log(\xi_1) =  - 2\Delta \sum_{j\neq n}\lambda_j-2\theta_1.  
\end{equation*}
\end{small}
Employing \eqref{eq:EigDiff}, we obtain
\begin{small}
\begin{equation*}
\begin{array}{lll}
 &\theta_1 \geq  \sum_{j\neq n}-\log \left(1-e^{-\Delta \Upsilon ((N_1+1)^2-j^2)}  \right).
\end{array}
\end{equation*}
\end{small}
Let $\ell := \log \left(1-e^{-\Delta \Upsilon (N_1+1-n)(2N_1+1)}  \right)$. Then, we have
\begin{small}
\begin{equation*}
\begin{array}{lll}
&\hspace{-3mm}\sum_{j\neq n}-\log \left(1-e^{-\Delta \Upsilon ((N_1+1)^2-j^2)}  \right)-\ell \\
&\hspace{-3mm} \geq  \sum_{j=1}^{N_1}-\log \left(1-e^{-\Delta \Upsilon j(2N_1+1)}  \right)\geq \int_1^{N_1+1}\mathcal{Q}_{\Delta \Upsilon (2N_1+1)}(x)\mathrm{d}x  \\
&\hspace{-3mm}= \mathcal{J}_{1,N_1+1}(\Delta \Upsilon (2N_1+1)) \geq \mathcal{J}_{1,2}(\Delta \Upsilon (2N_1+1)),
\end{array}
\end{equation*}
\end{small}
where the first inequality holds because $(N_1+1)^2-j^2  \leq (N_1+1-j)(2N_1+1)$ and the second holds because the sum in the second row can be viewed as Riemannian sum of the positive and monotonically decreasing function $\mathcal{Q}_{\Delta \Upsilon (2N_1+1)}(x)$  over $x\in [1,N_1]$. Hence, the integral provides a lower bound for the sum. The upper bound is proved analogously, using $(N_1+1-j)(N_1+1)\leq (N_1+1)^2-j^2 $.
\end{proof}


Consider the Lagrange polynomials $\left\{L_{\Phi,n}\right\}_{n\in[N_1]}$ as in \eqref{eq:Lagrange}.
\begin{lemma}\label{Lem:LagrangePBound}
For $n\in [N_1]$, we have
\begin{small}
\begin{equation}\label{eq:LPhiSqBound}
 \hspace{-2mm} L_{\Phi,n}^2 (\phi_{N_1+1}) = \begin{cases}e^{-2\Delta \sum_{j=n+1}^{N_1}(\lambda_j-\lambda_n)+\Theta}, \ n<N_1,\\
e^{2\theta_2-2\theta_1}, \ n=N_1,
 \end{cases}
\end{equation}
\end{small}
where 
\begin{small}
\begin{equation}\label{eq:ThetaDef}
\Theta = \begin{cases} -2\left(\theta_1-\theta_2-\theta_3\right), \quad n>1,\\
-2\left(\theta_1 -\theta_3 \right), \quad n=1.
\end{cases}
\end{equation}
\end{small}
Moreover, fixing $n\in [N_1-1]$, $\underline{\Delta}>0$ and denoting 
\begin{small}
\begin{equation*}
\sigma(n,N_1) :=  \frac{N_1(N_1+1)(2N_1+1)}{6} - \frac{n(n+1)(2n+1)}{6} -(N_1-n)n^2,
\end{equation*}
\end{small}
there exists a constant $M_{\phi}=M_{\phi}(\underline{\Delta})>0$ such that 
\begin{small}
\begin{equation}\label{eq:nFixedBound}
L_{\Phi,n}^2(\phi_{N_1+1}) \leq M_{\phi} e^{-2\Delta \sigma(n,N_1)} ,\quad \Delta\geq  \underline{\Delta}.
\end{equation}
\end{small}
\end{lemma}

\begin{proof}
The equality \eqref{eq:LPhiSqBound} follows from Lemma \ref{lem:ComponentBounds0} and the fact that  $L_{\Phi,n}^2 (\phi_{N_1+1})  = \frac{\xi_1}{\xi_2 \cdot \xi_3}$. Now let $n\in [N_1-1]$. In view of Lemma \ref{lem:ComponentBounds0}, $\Theta$ in \eqref{eq:ThetaDef} is uniformly bounded in $\Delta \in [\underline{\Delta},\infty)$. Moreover,
\begin{small}
\begin{equation}\label{eq:Sumbound}
\begin{array}{lll}
&\hspace{-6mm} -2\Delta \sum_{j=n+1}^{N_1}(\lambda_j-\lambda_n) \overset{\eqref{eq:EigDiff}}{\leq} -2\Delta \upsilon \sigma(n,N_1)\leq -2\Delta \upsilon (2N_1-1),
\end{array}
\end{equation}
\end{small}
whereas for $\Delta \geq \underline{\Delta}$, we employ Lemma \ref{lem:ComponentBounds0} and Proposition \ref{prop:Ical} to obtain
\begin{small}
\begin{equation}\label{eq:Thetabound}
e^{\Theta}\leq e^{2\left(\mathcal{J}_{0,\infty}(2\upsilon \underline{\Delta})+\mathcal{J}_{0,\infty}(3\upsilon \underline{\Delta}) \right)}=:M_{\phi}
\end{equation}
\end{small}
for $n\in [N_1-1]$. \eqref{eq:nFixedBound} follows from \eqref{eq:LPhiSqBound},\eqref{eq:Sumbound} and \eqref{eq:Thetabound}.
\end{proof}
\begin{prop}\label{Prop:Z4bound}
The term $\xi_4$ in \eqref{eq:XiComp} satisfies
\begin{small}
\begin{equation}\label{eq:xi4bound}
\xi_4 \leq M_{\xi} \frac{e^{\Delta \lambda_{N_1}}}{\Delta}
\end{equation}
\end{small}
for some $M_{\xi}>0$ independent of $\Delta>0$.
\end{prop}
\begin{proof}
We write $\xi_4 = \xi_{4,1}+\xi_{4,2}$, where 
\begin{small}
\begin{equation*}
\begin{array}{lll}
&\xi_{4,1} = \sum_{k\in [n-1]}\frac{1}{|\phi_n-\phi_k|} \mbox{ and } \xi_{4,2}  =\sum_{k=n+1}^{N_1}\frac{1}{|\phi_n-\phi_k|}.
\end{array}
\end{equation*}
\end{small}
For $\xi_{4,1}$ with $n>1$, we have
\begin{small}
\begin{equation}\label{eq:Z4Bound}
\begin{array}{lll}    
&\hspace{-3mm}\xi_{4,1} \leq \frac{e^{\Delta \lambda_n}}{\Delta}\sum_{k=1}^{n-1}\frac{1}{\lambda_n-\lambda_k} {\leq }\frac{e^{\Delta \lambda_n}}{\Delta \upsilon} \sum_{k=1}^{n-1}\frac{1}{n^2-k^2}\leq \frac{e^{\Delta \lambda_{n}}}{\Delta \upsilon}\frac{\ln(2n)}{2n }. 
\end{array}
\end{equation}
\end{small}
where the first inequality follows from the application of Lagrange's theorem with the derivative computed at $\lambda_n$, the second follows from \eqref{eq:EigDiff}. The third inequality follows from comparison with the integral of the positive and monotonically increasing function $x\mapsto (n^2-x^2)^{-1}$ on $x\in[1,n-1]$. Analogously, for $n<N_1$ we obtain
\begin{small}
\begin{equation}\label{eq:Z5Bound}
\begin{array}{lll}
 &\hspace{-4mm}\xi_{4,2}\leq \Delta^{-1}\sum_{k=n+1}^{N_1}\frac{e^{\Delta \lambda_k}}{\lambda_k-\lambda_n}\overset{\eqref{eq:EigDiff}}{\leq}\frac{e^{\Delta \lambda_{N_1}}}{\Delta \upsilon}\sum_{k=n+1}^{N_1}\frac{1}{k^2-n^2}\leq \frac{e^{\Delta \lambda_{N_1}}}{\Delta \upsilon}\frac{1+\ln(2n+1)}{2n}.
\end{array}
\end{equation}
\end{small}
The result follows from \eqref{eq:Z4Bound}, \eqref{eq:Z5Bound} since $\frac{\ln(2n)}{2n}$ and $\frac{1+\ln(2n+1)}{2n}$ are bounded sequences.
\end{proof}
We can now establish the asymptotic behaviour of the condition numbers $ \mathcal{K}_{\lambda}(n)$ and $\mathcal{K}_{y}(n)$ in the regime \eqref{eq:Regimes}.
\begin{theorem}\label{Thm:FirstOrdCond}
Recall the first-order condition numbers $\mathcal{K}_{\lambda}(n)$ and $\mathcal{K}_{y}(n)$ in \eqref{eq:EpsExpans}. Let $n\in \mathbb{N}$ and $\rho>0$. There exists $N_1^*(\rho)\in \mathbb{N}$ with $n<N_1^*(\rho)$ such that, for all $N_1>N_1^*(\rho)$,
\begin{small}
\begin{equation}\label{eq:AsympKappa}
    \begin{aligned}
    &\left|\mathcal{K}_{y}(n)\right| \leq \zeta_y(n,N_1,y_{N_1+1})\Delta^{-1}e^{-\rho \Delta},\\
    &\left|\mathcal{K}_{\lambda}(n)\right| \leq \zeta_{\lambda}\Delta^{-1}e^{-\rho \Delta},
    \end{aligned}
    \quad \Delta \to \infty
\end{equation}
\end{small}
with $\zeta_y(\lambda_{N_1},y_{N_1+1})>0\;,\zeta_{\lambda}>0$ independent of $\Delta$.
\end{theorem}
\begin{proof}
Given $\rho>0$, let $N_1^*(\rho)>n$ be large enough such that for all $N_1>N_1^*(\rho)$, $\lambda_{N_1}>0$ and
\begin{small}
\begin{equation}\label{eq:N1Cond}
-2 \upsilon \sigma(n,N_1)+\lambda_{N_1}<-\rho.   
\end{equation}
\end{small}
Note that such $N_1^*(\rho)>0$ exists since $\sigma(n,N_1)$ grows as $N_1^3$, whereas by \eqref{eq:EigBounds}, we have $\lambda_{N_1}=O(N_1^2)$ as $N_1\to \infty$. Consider first $\mathcal{K}_{y}(n)$. In view of \eqref{eq:Deriv}, we have
\begin{small}
\begin{equation*}
\begin{array}{lll}
\left|\mathcal{K}_{y}(n)\right| \overset{\text{Remark }  \ref{rem:Remark4}}{\leq} \left|y_{N_1+1} \right| \left(1+2e^{-\Delta \lambda_n} \xi_4\right)L_{\phi,n}^2(\phi_{N_1+1}).
\end{array}
\end{equation*}
\end{small}
Employing \eqref{eq:nFixedBound}, \eqref{eq:xi4bound} and \eqref{eq:N1Cond} with $\Delta\geq \underline{\Delta }>0$, 
\begin{small}
\begin{equation*}
\begin{array}{lll}
&\left|\mathcal{K}_{y}(n)\right|  \leq M_{\phi}\left|y_{N_1+1} \right| \left(\Delta+2 e^{-\Delta \lambda_n}M_{\xi} e^{\Delta \lambda_{N_1}}\right)\Delta^{-1}e^{-2\Delta \upsilon \sigma(n,N_1)} \\
&\leq M_{\phi}\left|y_{N_1+1} \right| \left(\Delta e^{-\Delta \lambda_{N_1}}+2 M_{\xi}e^{-\Delta \lambda_n}\right)\Delta^{-1}e^{-\rho \Delta}.
\end{array}
\end{equation*}
\end{small}
Since $\lambda_{N_1}>0$, $\max_{\Delta \geq \underline{\Delta }}\Delta e^{-\Delta \lambda_{N_1}}<\infty$, whence we can take 
\begin{small}
\begin{equation*}
\zeta_y(n,N_1,y_{N_1+1}): =M_{\phi}\left|y_{N_1+1} \right| \left(\max_{\Delta \geq \underline{\Delta}}\Delta e^{-\Delta \lambda_{N_1}}+2 M_{\xi}e^{-\Delta \lambda_n}\right)   
\end{equation*}
\end{small}
in \eqref{eq:AsympKappa}. Similarly, we have 
\begin{small}
\begin{equation*}
\begin{array}{lll}
\left|\mathcal{K}_{\lambda}(n)\right| &= \frac{\left|y_{N_1+1}\right|}{\Delta |y_n|}\frac{\left| \phi_{N_1+1}-\phi_{n}\right|}{\phi_n}L_{\Phi,n}^2\left(\phi_{N_1+1}\right).
\end{array}
\end{equation*}
\end{small}
By Assumptions \ref{assum:coeff}-\ref{assum:meas}, $\frac{\left|y_{N_1+1}\right|}{ |y_n|}\leq M_y$, whereas $\frac{\left| \phi_{N_1+1}-\phi_{n}\right|}{\phi_n}\leq 1$. Hence, from \eqref{eq:nFixedBound}, we again have
\begin{small}
\begin{equation*}
\left|\mathcal{K}_{\lambda}(n)\right| \leq M_y M_{\phi} \Delta^{-1}  e^{-2\Delta \sigma(n,N_1)} =: \zeta_{\lambda}\Delta^{-1}e^{-\rho \Delta},
\end{equation*}
\end{small}
which concludes the proof.
\end{proof}
\begin{remark}
Theorem \ref{Thm:FirstOrdCond} guarantees exponential decay of $\mathcal{K}_y(n)$ and $\mathcal{K}_{\lambda}(n)$ provided that $N_1>N_1^*(\rho)$. This condition is required only for $\mathcal{K}_y(n)$ and follows from the fact that the independent estimates of $L^2_{\Phi,n}(\phi_{N_1+1})$ and $|L_{\Phi,n}'(\phi_n)| = \xi_4$ in \eqref{eq:nFixedBound} and \eqref{eq:xi4bound}, respectively, are combined in bounding $\mathcal{K}_y(n)$, which leads to conservatism.  The numerical simulations in the next section show that the predicted exponential decay rate is obtained without choosing $N_1$ according to \eqref{eq:N1Cond}.
\end{remark}
\begin{remark}
Note that according to Theorem \ref{Thm:FirstOrdCond}, a larger value of the constant $N_1>N_1^*(\rho)$ actually allows one to obtain a larger constant $\rho$, thereby leading to faster decay of the condition numbers as $\Delta \to \infty$.
\end{remark}

\section{Numerical Simulations}\label{Sec:NumEx}

We provide numerical examples to validate our theory. All simulations are implemented in Wolfram Mathematica \cite{Mathematica}.

\subsection{Condition numbers}\label{sec:cond-numerics}
Consider the model \eqref{eq:Meas2} with $N_1=4$, $N_2=1$, $\lambda_n=n^2$, $y_n=1$ for all $n$ and $t_k=k\Delta$. We use \eqref{eq:Deriv1} to compute the condition numbers $\mathcal{K}_{\lambda}(n)$ and $\mathcal{K}_{y}(n), \ n\in [N_1]$. The results in Fig.~\ref{fig:kappa-N1-fixed} show exponential decay of the condition numbers, with the lowest order parameters ($n=1$) being the most stable, while the highest order ones ($n=N_1$) are the least stable. Differently from Theorem \ref{Thm:FirstOrdCond}, here we fix $N_1$ a priori, and still obtain an exponential decay of the condition numbers.

\subsection{ESPRIT algorithm}
\label{sec:esprit}
The ESPRIT algorithm \cite{roy1989esprit} is one of the best-performing methods for exponential fitting. It requires at least $2N_1$ equispaced samples of the signal $y(t)$ of the form \eqref{eq:Meas2}, and produces estimates of the parameters $\{\lambda_n,y_n\}_{n=1}^{N_1}$. It is known to provide exact solutions in the noiseless case (i.e. $\epsilon=0$), and performs close to optimal in the presence of noise, in the context of the so-called super-resolution problem in applied harmonic analysis \cite{batenkov2021a}. 

We apply ESPRIT to the sequence $\left\{y(\Delta k)\right\}_{k=0}^{2N_1-1}$, with the same setup as in Section \ref{sec:cond-numerics}. In Fig.~\ref{fig:esprit}, we see that the conditioning of the ESPRIT algorithm is  consistent with Theorem \ref{Thm:FirstOrdCond} and the computed condition numbers in Section~\ref{sec:cond-numerics}. We plot the rescaled errors (recall \eqref{eq:EpsExpans}),
\begin{small}
\begin{equation}\label{eq:rescaled-errors}
\mathcal{E}_{\lambda}(n) = \frac{\bigl|\widetilde{\lambda}_n-\lambda_n\bigr|}{\epsilon},\qquad \mathcal{E}_{y}(n) = \frac{\bigl|\widetilde{y}_n-y_n\bigr|}{\epsilon},
\end{equation}
\end{small}
where $\widetilde{\lambda}_n$ and $\widetilde{y}_n$ are the parameter values recovered by ESPRIT,
and, furthermore, the $\widetilde{\lambda}_n$'s have been index-matched to the true $\lambda_n$'s. Here $\epsilon=10^{-1}$, and the results were computed with $32$ decimal digits of precision.

\begin{figure}
    \begin{subfigure}{0.49\columnwidth}
        \includegraphics[width=\linewidth]{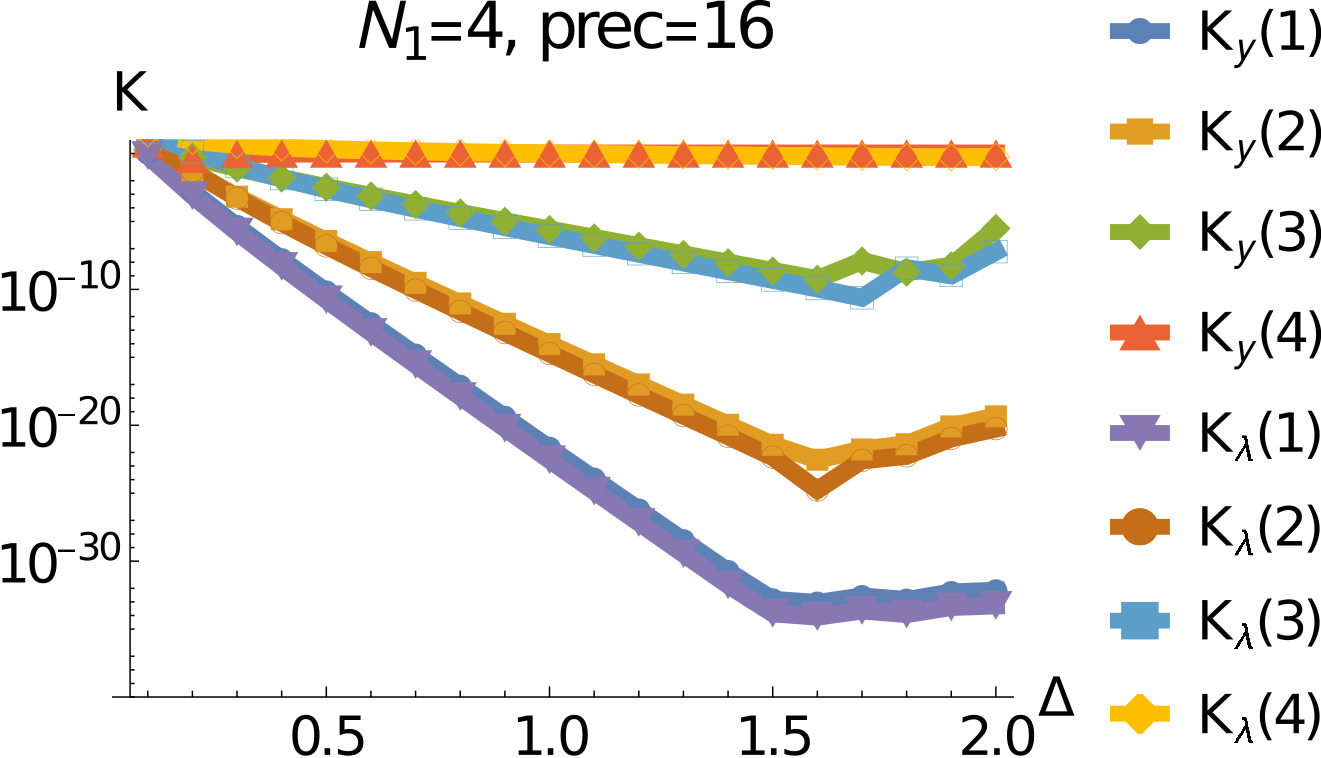}
        \caption{}
    \label{fig:kappa-N1-fixed}
    \end{subfigure}  
    \begin{subfigure}{0.49\columnwidth}
        \includegraphics[width=\linewidth]{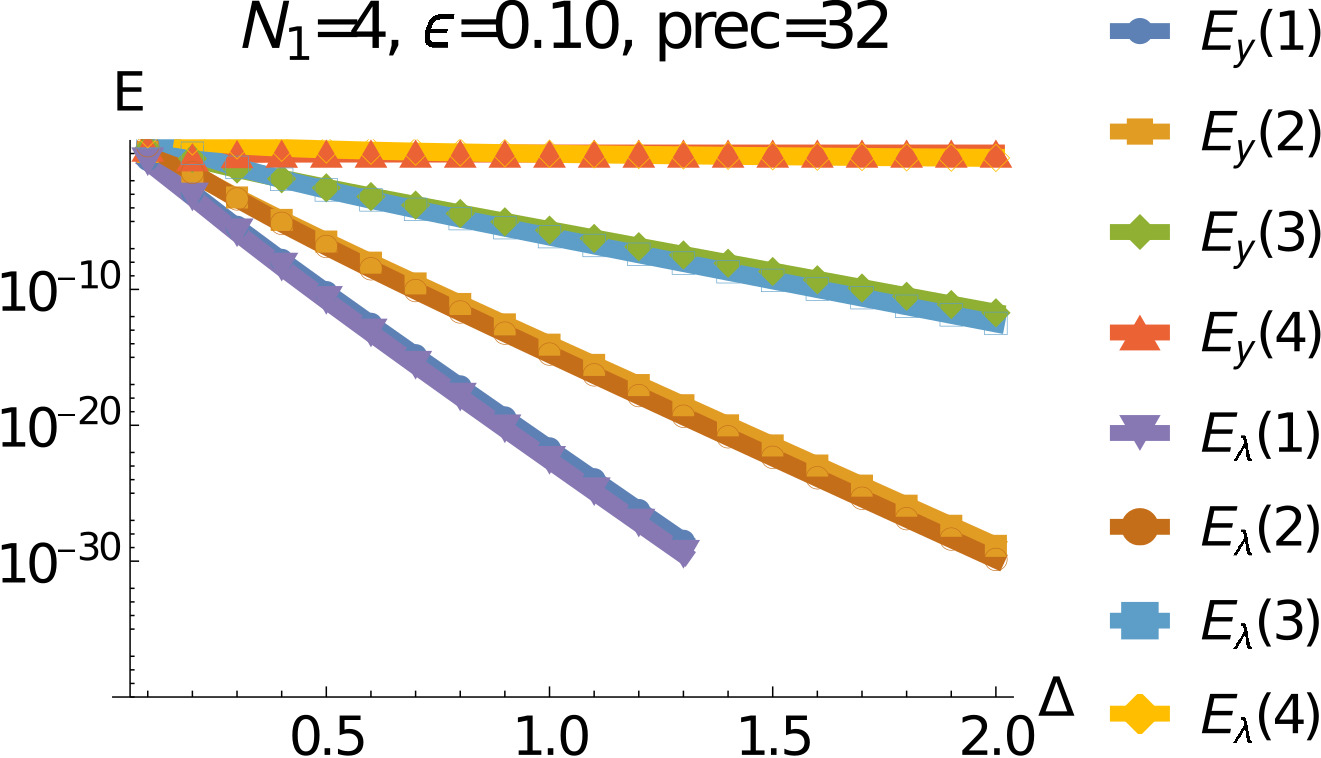}  
        \caption{}  
    \label{fig:esprit}
    \end{subfigure} 
    \caption{(\ref{fig:kappa-N1-fixed})  $\mathcal{K}_\lambda,\mathcal{K}_{y}$ for  $N_1=4$ and $\lambda_n=n^2$.  The asymptotics break down for large $\Delta$, due to inversion of badly conditioned matrices in finite precision computations (16 decimal digits). (\ref{fig:esprit}) ESPRIT algorithm conditioning (see  \eqref{eq:rescaled-errors}), applied to the sequence $\left\{y(\Delta k)\right\}_{k=0}^{2N_1-1}$ with $N_1=4$ and $\lambda_n=n^2$. Here we used 32 decimal digits of precision.}
\end{figure}

\subsection{PDE parameter identification}
\label{sec:pde-identification}

We test the complete procedure on a PDE identification problem. We consider the PDE \eqref{eq:PDE} with constant $p\equiv q \equiv 0.1$. The eigenvalues and eigenfunctions are explicitly given by $\lambda_n = n^2\pi^2-q$ and $\psi_n(x) = \sqrt{2} \sin(n\pi x)$. The initial condition is set to be $z_n(0)=(-1)^{n+1} \left( \sqrt{2} n^3\right)^{-1}$.
To solve the PDE, we use the method of lines for space discretization with $N_x=60$ collocation points and 4th order finite difference approximation, and the resulting ODE system is integrated for $t\in[0,2]$. The resulting solution and the initial condition are plotted in Fig.~\ref{fig:pde-solution}. Our implementation utilized the \texttt{NDSolve} library function.

Next, let $c(x)=\sum_{n=1}^{N_1} c_n \psi_n(x)+\epsilon \sum_{n=N_1+1}^{N_1+2} c_n \psi_n(x)$ where $\{c_n\}_{n=1}^{N_1+2}$ with $c_n\in[1,2]$ are randomly chosen, and $\epsilon=10^{-4}$. The measurements $y(t)$ in \eqref{eq:Meas} are computed using global adaptive quadrature as implemented in \texttt{NIntegrate} library function. Finally, $y(t)$ is sampled at $1025$ equispaced points in $[0,2]$, thus giving a minimal value of $\Delta_{\min}:=\frac{1}{512}$. The filter and the sampled measurements are shown in Fig.~\ref{fig:pde-measurements}.

\begin{figure}
    \begin{center}
        \includegraphics[width=\linewidth]{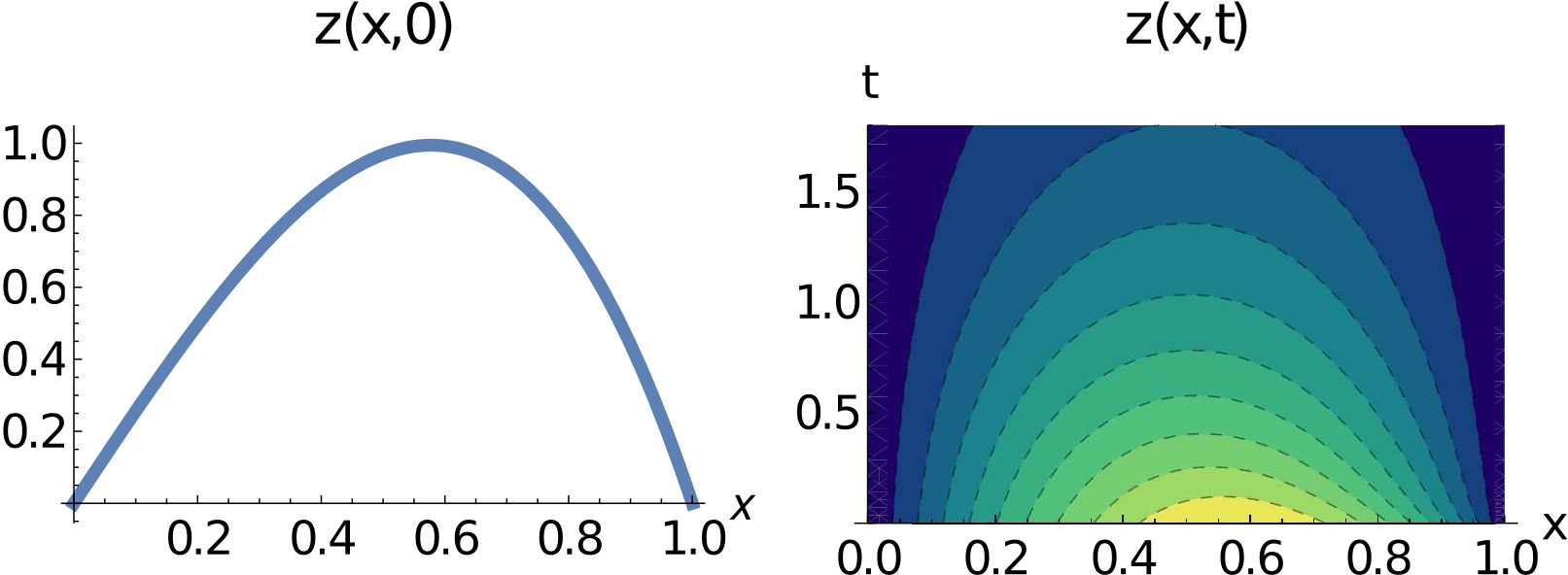}
    \end{center}
    \caption{Numerical solution of the PDE (right) with the specified initial condition (left).}
    \label{fig:pde-solution}
\end{figure}

\begin{figure}
    \begin{center}
        \includegraphics[width=\linewidth]{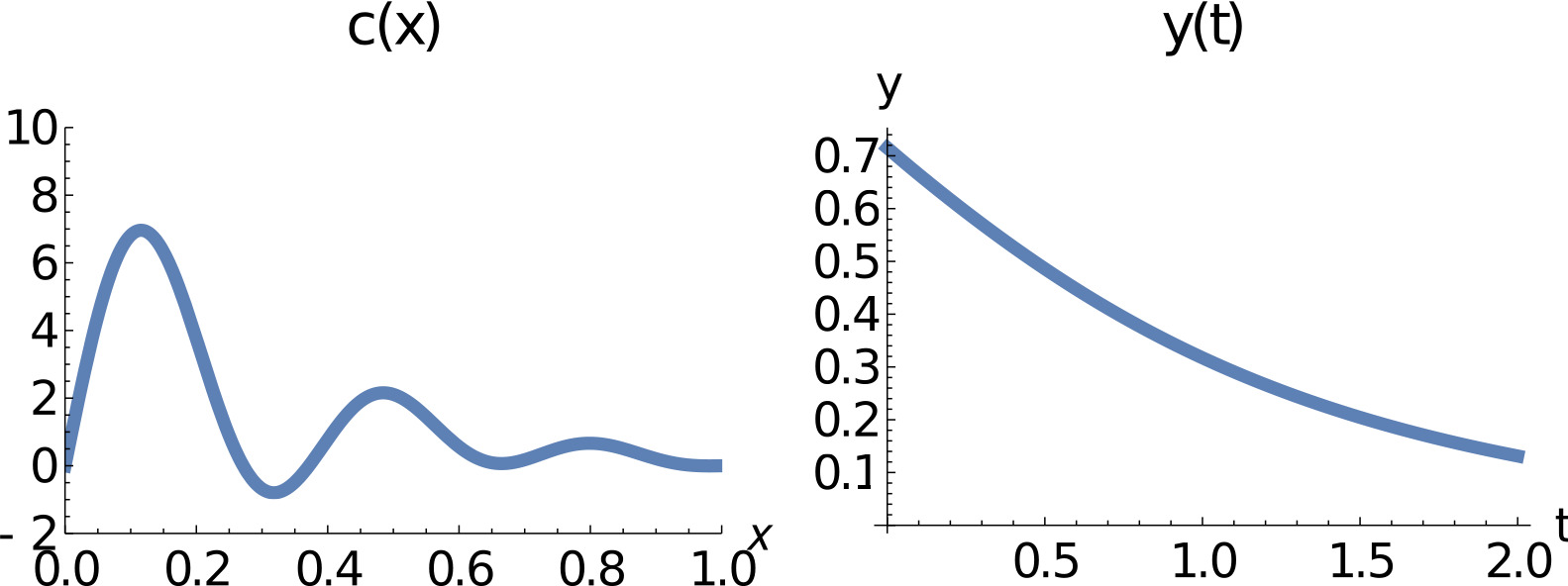}
    \end{center}
    \caption{The measurement filter $c(x),\;x\in[0,1]$ and the corresponding non-local measurement data $y(t),\;t\in[0,2]$.}
    \label{fig:pde-measurements}
\end{figure}

\begin{figure}[h!]
    \begin{subfigure}{0.49\columnwidth}
        \includegraphics[width=\linewidth]{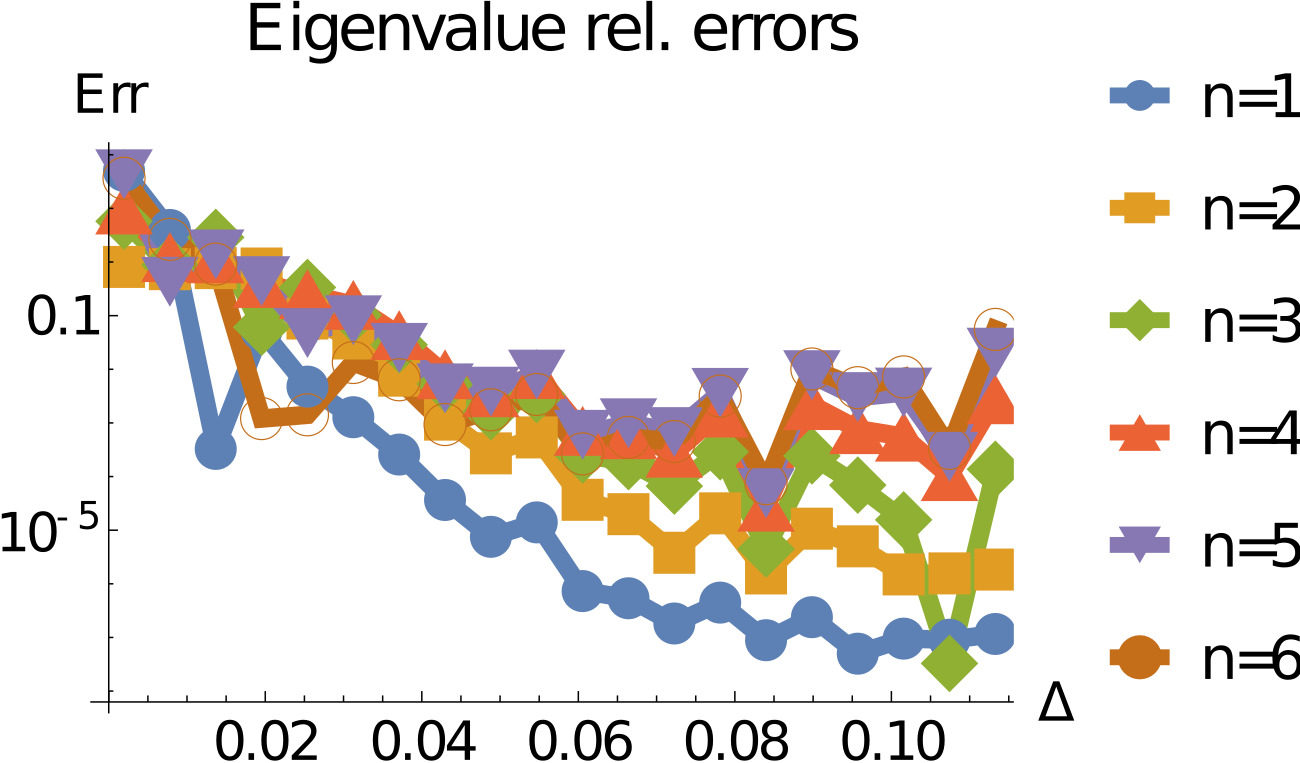}
        \caption{}
        \label{fig:pde-esprit-node-errors}
    \end{subfigure}
    \hfill
    \begin{subfigure}{0.49\columnwidth}
        \includegraphics[width=\linewidth]{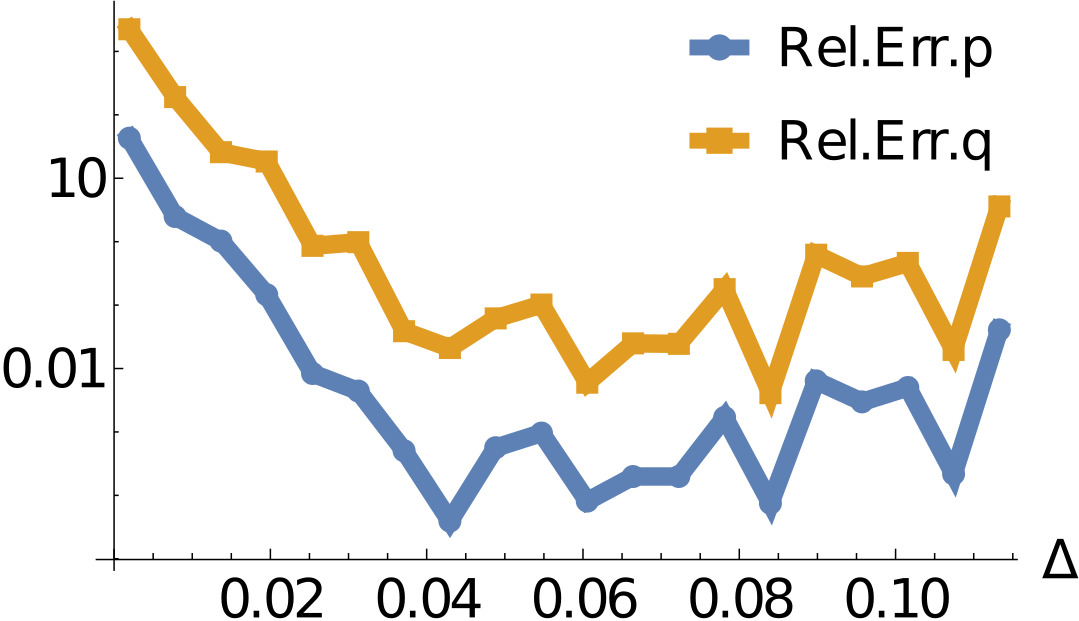}
        \caption{}
        \label{fig:pde-esprit-pq}
    \end{subfigure}
    \caption{ESPRIT recovery errors. (\ref{fig:pde-esprit-node-errors}) Recovery of PDE eigenvalues by ESPRIT, increasing $\Delta$. The relative errors in ${\lambda_n}$ for $n=1,\dots,N_1$ are plotted. (\ref{fig:pde-esprit-pq}) Recovery errors of $p,q$, estimated from $\{\widehat{\lambda}_n\}_{n=1}^{N_1}$ by a linear least squares fit.}
    \label{fig:esprit-errors}
\end{figure}

We apply the ESPRIT algorithm on  $\{y(k\Delta)\}_{k=0}^{2N_1-1}$ with varying $\Delta$. The relative errors in the recovered eigenvalues are plotted in Fig.~\ref{fig:pde-esprit-node-errors}. The deterioration of the error when $\Delta$ passes a certain threshold is consistent with our earlier observations due to the finite precision in the computations. Here, all computations are done with $100$ decimal digits of precision. 
Finally, we estimate $p,q$ from the recovered eigenvalues, using the relationship $\lambda_n=\pi^2 n^2 p -q$ by applying linear least squares regression to $\{\widehat{\lambda}_n\}_{n=1}^{N_1}$. The errors in the estimated parameters are plotted in Fig.~\ref{fig:pde-esprit-pq}.

\bibliographystyle{IEEEtran}
\bibliography{IEEEabrv,BibliographyKSE}

\begin{thebibliography}{10}
\providecommand{\url}[1]{#1}
\csname url@samestyle\endcsname
\providecommand{\newblock}{\relax}
\providecommand{\bibinfo}[2]{#2}
\providecommand{\BIBentrySTDinterwordspacing}{\spaceskip=0pt\relax}
\providecommand{\BIBentryALTinterwordstretchfactor}{4}
\providecommand{\BIBentryALTinterwordspacing}{\spaceskip=\fontdimen2\font plus
\BIBentryALTinterwordstretchfactor\fontdimen3\font minus
  \fontdimen4\font\relax}
\providecommand{\BIBforeignlanguage}[2]{{%
\expandafter\ifx\csname l@#1\endcsname\relax
\typeout{** WARNING: IEEEtran.bst: No hyphenation pattern has been}%
\typeout{** loaded for the language `#1'. Using the pattern for}%
\typeout{** the default language instead.}%
\else
\language=\csname l@#1\endcsname
\fi
#2}}
\providecommand{\BIBdecl}{\relax}
\BIBdecl

\bibitem{sivashinsky1977nonlinear}
G.~Sivashinsky, ``Nonlinear analysis of hydrodynamic instability in laminar
  flames--{I. Derivation} of basic equations,'' \emph{Acta Astronautica},
  vol.~4, pp. 1177--1206, 1977.

\bibitem{nicolaenko1986some}
B.~Nicolaenko, ``Some mathematical aspects of flame chaos and flame
  multiplicity,'' \emph{Physica D: Nonlinear Phenomena}, vol.~20, no.~1, pp.
  109--121, 1986.

\bibitem{balas1988finite}
M.~J. Balas, ``Finite-dimensional controllers for linear distributed parameter
  systems: exponential stability using residual mode filters,'' \emph{Journal
  of Mathematical Analysis and Applications}, vol. 133, no.~2, pp. 283--296,
  1988.

\bibitem{harkort2011finite}
C.~Harkort and J.~Deutscher, ``Finite-dimensional observer-based control of
  linear distributed parameter systems using cascaded output observers,''
  \emph{International Journal of Control}, vol.~84, no.~1, pp. 107--122, 2011.

\bibitem{katz2022delayed}
R.~Katz and E.~Fridman, ``Delayed finite-dimensional observer-based control of
  {1D} parabolic {PDEs} via reduced-order {LMIs},'' \emph{Automatica}, vol.
  142, p. 110341, 2022.

\bibitem{christofides2001}
P.~Christofides, \emph{Nonlinear and Robust Control of PDE Systems: Methods and
  Applications to transport reaction processes}.\hskip 1em plus 0.5em minus
  0.4em\relax Springer, 2001.

\bibitem{curtain1982finite}
R.~Curtain, ``Finite-dimensional compensator design for parabolic distributed
  systems with point sensors and boundary input,'' \emph{IEEE Transactions on
  Automatic Control}, vol.~27, no.~1, pp. 98--104, 1982.

\bibitem{katz2021finite111}
R.~Katz and E.~Fridman, ``{Finite-dimensional boundary control of the linear
  Kuramoto-Sivashinsky equation under point measurement with guaranteed $
  L^{2}$-gain},'' \emph{IEEE Transactions on Automatic Control}, vol.~67,
  no.~10, pp. 5570--5577, 2021.

\bibitem{demetriou1994dynamic}
M.~Demetriou and I.~Rosen, ``Dynamic identification of implicit parabolic
  systems,'' \emph{Lecture Notes in Pure and Applied Mathematics}, pp.
  153--153, 1994.

\bibitem{banks2012estimation}
H.~T. Banks and K.~Kunisch, \emph{Estimation techniques for distributed
  parameter systems}.\hskip 1em plus 0.5em minus 0.4em\relax Springer Science
  \& Business Media, 2012.

\bibitem{lowe1992recovery}
B.~D. Lowe, M.~Pilant, and W.~Rundell, ``The recovery of potentials from finite
  spectral data,'' \emph{SIAM Journal on Mathematical Analysis}, vol.~23,
  no.~2, pp. 482--504, 1992.

\bibitem{rundell1992reconstruction}
W.~Rundell and P.~E. Sacks, ``Reconstruction techniques for classical inverse
  {Sturm-Liouville} problems,'' \emph{Mathematics of Computation}, vol.~58, no.
  197, pp. 161--183, 1992.

\bibitem{kirsch2011introduction}
A.~Kirsch, \emph{An introduction to the mathematical theory of inverse
  problems}.\hskip 1em plus 0.5em minus 0.4em\relax Springer, 2011, vol. 120.

\bibitem{istratov1999a}
A.~A. Istratov and O.~F. Vyvenko, ``Exponential analysis in physical
  phenomena,'' \emph{Review of Scientific Instruments}, vol.~70, no.~2, pp.
  1233--1257, 1999.

\bibitem{pereyra2010}
V.~Pereyra and G.~Scherer, \emph{Exponential {{Data Fitting}} and {{Its
  Applications}}}.\hskip 1em plus 0.5em minus 0.4em\relax {Bentham Science
  Publishers}, 2010.

\bibitem{batenkov2021a}
D.~Batenkov, G.~Goldman, and Y.~Yomdin, ``Super-resolution of near-colliding
  point sources,'' \emph{Information and Inference: A Journal of the IMA},
  vol.~10, no.~2, pp. 515--572, 2021.

\bibitem{roy1989esprit}
R.~Roy and T.~Kailath, ``{ESPRIT}-estimation of signal parameters via
  rotational invariance techniques,'' \emph{IEEE Transactions on Acoustics,
  Speech, and Signal Processing}, vol.~37, no.~7, pp. 984--995, 1989.

\bibitem{orlov2017general}
Y.~Orlov, ``On general properties of eigenvalues and eigenfunctions of a
  {S}turm-{L}iouville operator: comments on “{ISS} with respect to boundary
  disturbances for {1-D} parabolic {PDEs}”,'' \emph{IEEE Transactions on
  Automatic Control}, vol.~62, no.~11, pp. 5970--5973, 2017.

\bibitem{katz2021finite}
R.~Katz and E.~Fridman, ``Finite-dimensional control of the heat equation:
  Dirichlet actuation and point measurement,'' \emph{European Journal of
  Control}, vol.~62, pp. 158--164, 2021.

\bibitem{fulton1994eigenvalue}
C.~T. Fulton and S.~A. Pruess, ``Eigenvalue and eigenfunction asymptotics for
  regular {S}turm-{L}iouville problems,'' \emph{Journal of Mathematical
  Analysis and Applications}, vol. 188, no.~1, pp. 297--340, 1994.

\bibitem{spivak2018calculus}
M.~Spivak, \emph{Calculus on manifolds: a modern approach to classical theorems
  of advanced calculus}.\hskip 1em plus 0.5em minus 0.4em\relax CRC press,
  2018.

\bibitem{Mathematica}
\BIBentryALTinterwordspacing
{Wolfram Research{,} Inc.}, ``Mathematica, {V}ersion 14.0,'' {Champaign, IL,
  2024}. [Online]. Available: \url{https://www.wolfram.com/mathematica}
\BIBentrySTDinterwordspacing

\end{thebibliography}

\end{document}